\documentclass[english,10pt]{article}
\usepackage{graphicx}
\usepackage{amsmath}
\usepackage{amssymb}
\usepackage{pst-node}
\usepackage{color}
\usepackage{graphicx}
\usepackage{epstopdf}
\usepackage{tikz}
\usepackage{pgfplots}
\usepackage{pgflibraryarrows}
\usepackage{pgflibrarysnakes}
\usepackage{ifpdf}

\RequirePackage{marginnote,hyperref}

\newtheorem{theorem}{Theorem}[section]
\newtheorem{corollary}[theorem]{Corollary}
\newtheorem{definition}[theorem]{Definition}
\newtheorem{conjecture}[theorem]{Conjecture}

\newtheorem{problem}[theorem]{Problem}

\newtheorem{lemma}[theorem]{Lemma}

\newtheorem{claim}{Claim}
\newtheorem{proposition}[theorem]{Proposition}
\newenvironment{proof}{\noindent {\bf Proof.}}{\rule{3mm}{3mm}\par\medskip}

\def\qed{\mbox{\rule[0.3mm]{2ex}{2ex}}} 

\newcommand{\F}{\mbox{$\cal F$}}

\newcommand{\T}{\mbox{$\cal T$}}
\newcommand{\TP}{\mbox{$\mathcal P$}}

\newcommand{\Z}{\mbox{$\mathbb Z$}}

\newcommand{\setCZ}{\mbox{${\mathcal S}_3$}}

\newcommand{\setZTHR}{\mbox{$\langle\Z_3\rangle$}}

\title{Spanning Triangle-trees and Flows of Graphs\footnote{Supported by NSFC No.11871034, 11531011 and NSFQH No.2017-ZJ-790.}}

\date{}
\author{
\small  Jiaao Li$^1$, Xueliang Li$^2$, Meiling Wang$^2$\\
\small $^1$School of Mathematical Sciences and LPMC \\
\small Nankai University, Tianjin 300071, China \\
\small $^2$Center for Combinatorics and LPMC \\
\small Nankai University, Tianjin 300071, China \\
\small Emails: lijiaao@nankai.edu.cn; lxl@nankai.edu.cn; Estellewml@gmail.com
}

\begin{document}

\maketitle
\begin{abstract}
  In this paper we study the flow-property of graphs containing a spanning triangle-tree. Our main results provide a structure characterization of graphs with a spanning triangle-tree admitting a nowhere-zero $3$-flow. All these graphs without nowhere-zero $3$-flows are constructed from $K_4$ by a so-called bull-growing operation. This generalizes a result of Fan et al. in 2008 on triangularly-connected graphs and particularly shows that every $4$-edge-connected graph with a spanning triangle-tree has a nowhere-zero $3$-flow. A well-known classical theorem of Jaeger in 1979 shows that  every graph with two edge-disjoint spanning trees admits a nowhere-zero $4$-flow. We prove that every graph with two edge-disjoint spanning triangle-trees has a flow strictly less than $3$.
\\[2mm]
\\[2mm]\textbf{AMS Subject Classifications 2010:}  05C21, 05C40, 05C05
\end{abstract}

\section{Introduction}

We shall introduce some necessary notation and terminology and the concepts of $3$-flows, circular flows and group connectivity in the next subsections.

\subsection{The $3$-flows}

Graphs considered here may contain parallel edges, but no loops. We follow the textbook \cite{bondy} for undefined terminology and notation. For a graph $G$, we use $V(G)$ and $E(G)$ to denote the vertex set and edge set of $G$, respectively.
When $S$ is an edge subset of $E(G)$ or a vertex subset of $V(G)$, we use $G[S]$ to denote the edge-induced subgraph or the vertex-induced subgraph from $S$. For a vertex $u\in V(G)$, $d_G(u)$ denotes the degree of $u$ in $G$. Sometimes the subscript is omitted for convenience. We call $u$ a {\em $k$-vertex} ($k^+$-vertex, resp.) if $d(u)=k$ ($d(u)\ge k$, resp.). A $k$-cut is an edge-cut of size $k$. Let $D$ be an orientation of $G$.  The set of outgoing-arcs incident to $u$ is denoted by $E_D^+(u)$, while the set of incoming-arcs is denoted by $E_D^-(u)$. We use $d^+_D(v)=|E_D^+(u)|$, $d^-_D(v)=|E_D^-(u)|$ to denote the out-degree and in-degree of $u$, respectively.

Given an orientation $D$ and  a function $f$ from $E(G)$ to $\{\pm1,\pm2,\cdots,\pm (k-1)\}$, if $\sum_{e\in E_D^+(v)}f(e)=\sum_{e\in E_D^-(v)}f(e)$ for each vertex $v\in V(G)$, then we call $(D,f)$ a {\it nowhere-zero $k$-flow}, abbreviated as $k$-NZF. The flow theory was initiated by Tutte \cite{Tutte3flow_conj1&modk_flow}, generalizing face-colorings of plane graphs to flows of arbitrary non-planar graphs by  duality. Tutte proposed a well-known $3$-flow conjecture, which was selected by Bondy among the {\it Beautiful Conjectures in Graph Theory} \cite{bondybeauconj} with high evaluation.

\begin{conjecture}(Tutte's $3$-flow conjecture)\label{Tutteconj}
Every $4$-edge-connected graph has a $3$-NZF.
\end{conjecture}

Jaeger's $4$-flow theorem\cite{Jaeger_4_flow_thm} in 1979 shows that every $4$-edge-connected graph admits a nowhere-zero $4$-flow. This theorem was proved from spanning trees to finding even subgraph covers, and a stronger version concerning spanning trees is as follows.

\begin{theorem}\cite{Jaeger_4_flow_thm}\label{4flow}
Every graph with two edge-disjoint spanning trees has a $4$-NZF.
\end{theorem}

For graphs with higher edge-connectivity, breakthrough results for Conjecture \ref{Tutteconj} were obtained by Thomassen \cite{Thomassen conjecture_JCTB} and Lov\'{a}sz, Thomassen, Wu and Zhang  \cite{Lovasz}, which eventually confirmed Conjecture \ref{Tutteconj} for $6$-edge-connected graphs.

\begin{theorem}\cite{Lovasz}\label{THM:LTWZ2013}
Every $6$-edge-connected graph admits a $3$-NZF.
\end{theorem}

On the other hand, Kochol \cite{Koch01} proved that it suffices to prove Conjecture \ref{Tutteconj} for $5$-edge-connected graphs and he also showed that Conjecture \ref{Tutteconj} is equivalent to the statement that every bridgeless graph with at most three $3$-cuts admits a $3$-NZF. There are infinite many graphs with exactly four $3$-cuts but admitting no $3$-NZF. Several such graph families were given in \cite{triconZ_3,groupconnLai_complefamil,Lai-Locally}. Most of these graphs consist of $2$-sums of $K_4$ (defined later), and majority of whose edges lie in triangles. This may suggest that the potential minimal counterexamples of Conjecture \ref{Tutteconj} (or its equivalent form) may contain many triangles. For more examples, see \cite{Dvorak} which characterizes all planar non vertex-$3$-colorable graphs with four triangles, whose duals also contain similar structures.

A graph is {\em triangular} if each edge is contained in a triangle $K_3$. Xu and Zhang \cite{XZ_CONJECTURE} suggested to consider Conjecture \ref{Tutteconj} for triangular graphs and they verified Conjecture \ref{Tutteconj} for squares of graphs, a subclass of triangular graphs.   Other examples of triangular graphs are the triangulations on surfaces, chordal graphs and locally connected graphs, whose flow-property was studied in \cite{surface,groupconnLai_complefamil,Lai-Locally}, among others.

\begin{definition}
A {\bf triangle-tree} ${\mathcal T}(x_1,x_2,\ldots,x_n)$ is formed by starting with a triangle $x_1x_2x_3$ and then repeatedly adding vertices in such a way that each added vertex $x_{j+1}$ is connected to exactly two adjacent vertices $y,z$ in ${\mathcal T}(x_1,x_2,\ldots,x_j)$ such that, together, the vertices $x_{j+1},y,z$ form a triangle. A $2$-vertex in the triangle-tree is called a {\bf leaf}.
For $n\ge 4$, a {\bf triangle-path} ${\mathcal P}(x_1,x_2,\ldots,x_n)$ is a triangle-tree with precisely two leaves. In the trivial case $n=3$, ${\mathcal P}(x_1,x_2,x_3)$ is a triangle, also considered as a trivial triangle-path.

A graph $G$ is {\bf triangularly-connected} if for any pair of edges $e_1,e_2\in E(G)$, there is a triangle-path containing $e_1$ and $e_2$.
\end{definition}

The above-mentioned graph classes presented in \cite{surface,groupconnLai_complefamil,Lai-Locally,XZ_CONJECTURE} are all triangularly-connected. Fan et al. \cite{triconZ_3} obtained a complete characterization of triangularly-connected graphs with $3$-NZF using $2$-sum operation. Let $A$, $B$ be two subgraphs of $G$. We call $G$ the {\bf $2$-sum} of $A$ and $B$, denoted by $G=A\bigoplus_2 B$, if $E(G)=E(A)\bigcup E(B)$, $|E(A)\bigcap E(B)|=1$ and $|V(A)\bigcap V(B)|=2$. The wheel graph $W_k$ is constructed by adding a new center vertex connecting to each vertex of a $k$-cycle, where $k\ge 3$. A wheel $W_k$ is odd (even, resp.) if $k$ is an odd (even, resp.) number. Note that $K_4$ is also viewed as the odd wheel $W_3$.

\begin{theorem}{\em (Fan, Lai, Xu, Zhang, Zhou \cite{triconZ_3})} \label{triconn3flow}
Let $G$ be a triangularly-connected graph. Then $G$ has no
$3$-NZF if and only if there is an odd wheel $W$ and a subgraph $G_1$ such that $G = W \bigoplus_2 G_1$, where $G_1$ is a triangularly-connected graph without $3$-NZF.
\end{theorem}

In this paper, we push further to study the $3$-flows of even wider graph class, i.e. graphs containing a spanning triangle-tree. Triangularly-connected graphs most likely contain a spanning triangle-tree, but not vice versa, as some edge(s) may not be contained even in any triangle, see Figs. \ref{longerpath} and \ref{Fig:3coloring} for instances. More detailed comparison of these two graph classes is discussed in the last section.

As we need to handle certain $3$-connected graphs, the $2$-sum operation is not enough to achieve this work. We develop a new tool, called the {\it bull-growing/bull-reduction}.
Let $u,v$ be two adjacent $3$-vertices of a graph $G$ with a common neighbor $w$. The third neighbor of $u$ and $v$ is denoted by $a$ and $b$, respectively.  Let $H=G-u-v+ab$ (and we delete possible loops when $a=b$). Then $H$ is called the {\it bull-reduction} of $G$, and $G$ is a {\it bull-growing} of $H$ (see Fig. \ref{Fig:bullreduce}), and we write $G= {\mathcal B}\biguplus H$.

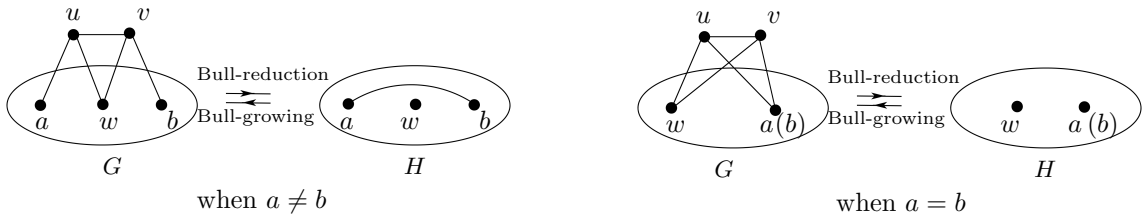
\begin{figure}[!htbp]
\centering
\minipage{0.45\textwidth}
  \includegraphics[width=1\textwidth]{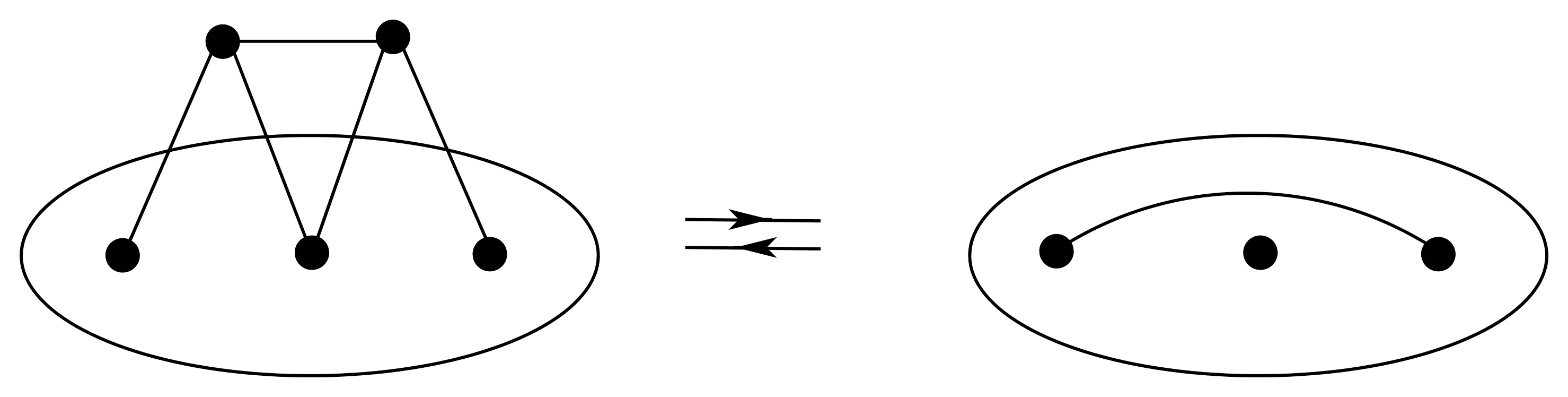}
  \put(-12,3.5){$u$}
  \put(-10.1,3.5){$v$}
  \put(-12.8,0.6){$a$}
  \put(-9.3,0.6){$b$}
  \put(-11.1,0.6){$w$}
  \put(-4.7,0.6){$a$}
  \put(-1,0.6){$b$}
  \put(-3.1,0.6){$w$}
  \put(-8.5,1.9){\scriptsize Bull-reduction}
  \put(-8.5,0.8){\scriptsize Bull-growing}
  \put(-11,-0.6){\small$G$}
  \put(-3,-0.6){\small$H$}
  \put(-8.5,-1.5){{when $a\neq b$}}
\endminipage\hfill
\minipage{0.45\textwidth}
  \includegraphics[width=\textwidth]{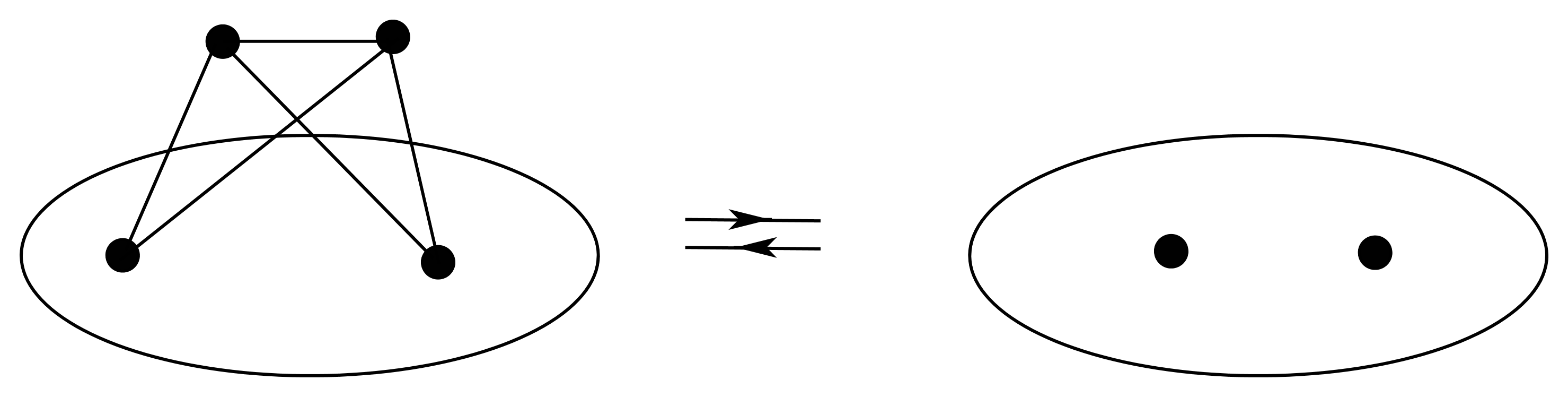}
  \put(-12,3.5){$u$}
  \put(-10.1,3.5){$v$}
  \put(-10.3,0.6){$a$}
  \put(-9.9,0.6){$(b)$}
  \put(-12.8,0.6){$w$}
  \put(-2.1,0.6){$a$}
  \put(-1.6,0.6){$(b)$}
  \put(-3.9,0.6){$w$}
  \put(-8.5,1.9){\scriptsize Bull-reduction}
  \put(-8.5,0.8){\scriptsize Bull-growing}
  \put(-11.5,-0.6){\small$G$}
  \put(-3,-0.6){\small$H$}
  \put(-8.5,-1.5){{ when $a=b$}}
\endminipage
\caption{Bull-reduction and bull-growing.}
\label{Fig:bullreduce}
\end{figure}

\begin{theorem}\label{mainthm3flow}
Let $G$ be a graph containing a spanning triangle-tree. Then $G$ has no
$3$-NZF if and only if $G = {\mathcal B} \biguplus G_1$, where $G_1$  contains a spanning triangle-tree and has no $3$-NZF. In other words, $G$ has no $3$-NZF if and only if $G$ is formed from $K_4$ by a series of bull-growing operations.
\end{theorem}

Since each step of the bull-growing operation on a graph does not decrease the number of $3$-vertices in the graph, we obtain a direct corollary of Theorem \ref{mainthm3flow}, verifying Conjecture \ref{Tutteconj} for those graphs in a strong sense.

\begin{corollary}\label{COR:3_3-vertices}
Every graph with a spanning triangle-tree has a $3$-NZF, provided that it contains at most three $3$-vertices.
\end{corollary}

\subsection{Circular Flows and Group Connectivity}

For integers $t\ge 2s>0$, a circular $t/s$-flow of a graph $G$ is a  $t$-NZF $(D,f)$ such that $s\le |f(e)|\le t-s$ for any edge $e\in E(G)$.   The flow index was defined in \cite{Goddyn1998} as the least rational number $r$ such that $G$ has a circular $r$-flow. Jaeger \cite{Jaeger} generalized Tutte's flow conjectures and proposed a conjecture that every $4k$-edge-connected graph admits a circular $(2+1/k)$-flow. It was confirmed for $6k$-edge-connected graph by Lov\'asz et al. \cite{Lovasz}, while eventually disproved in \cite{counterexample} for $k\ge 3$. But the cases for $k=1,2$ concerning $4$-, $8$-edge-connected graphs are  still particularly important since they imply Tutte's $3$-flow and $5$-flow conjectures, respectively. Closely related to those conjectures, the authors in \cite{flowindex} studied the problem of flow index less than $3$, sandwiched between $2.5$ and $3$. They proved that every $8$-edge-connected graph has a flow index strictly less than $3$, and conjectured that $6$-edge-connectivity suffices. Here we obtain a result for the flow index less than $3$ in the spirit of Theorem \ref{4flow}.

\begin{theorem}\label{THM:2tri-trees}
Every graph with two edge-disjoint spanning triangle-trees has a flow index strictly less than $3$.
\end{theorem}

Almost of all the above-mentioned flow results in fact use some orientation techniques. An orientation $D$ of $G$ is a {\it mod $k$-orientation} if for each vertex $v$ of $V(G)$, $d_D^+(v)-d_D^-(v)=0\pmod{k}$. The study of $3$-flows frequently uses mod $3$-orientation, since Tutte \cite{Tutte3flow_conj1&modk_flow} proved that a graph has a $3$-NZF if and only if it admits a mod $3$-orientation. This fact was generalized by Jaeger \cite{Jaeger} who showed that a graph has a circular $(2+1/p)$-flow if and only if it admits a mod $(2p+1)$-orientation. Moreover, it was proved in \cite{flowindex} that a connected graph has a flow index strictly less than $2+1/p$ if and only if it admits a strongly connected mod $(2p+1)$-orientation. Hence, we shall prove Theorem \ref{THM:2tri-trees} using strongly connected mod $3$-orientations.

Serving for a stronger induction process in proof, we will sometimes need certain orientation with prescribed boundaries, that is the concept of {\em group connectivity} introduced by Jaeger, Linial, Payan and Tarsi \cite{JLPT92}. A {\it $\Z_3$-boundary} $\beta$ of a graph $G$ is a mapping from $V(G)$ to $\Z_3$ with $\sum_{v\in V(G)}\beta(v)\equiv 0\pmod 3$. If for any $\Z_3$-boundary $\beta$, there is an orientation $D$ of $G$ such that $d_D^+(v)-d_D^-(v)\equiv\beta(v)\pmod{3}$ for any vertex $v\in V(G)$, then we say that $G$ is {\it $\Z_3$-connected}. Denote by  $\langle \Z_3\rangle$ the set of all the $\Z_3$-connected graphs. The advantage of this stronger property is to allow us to extend a mod $3$-orientation of $G/H$ to that of $G$ when the subgraph $H$ is $\Z_3$-connected (cf.\cite{JLPT92,groupconnLai_complefamil,Lovasz}). For strongly connected mod $3$-orientations, a similar property is defined in \cite{flowindex}. Let ${\mathcal S}_3$ be the family of all graphs $G$ such that for any $\Z_3$-boundary $\beta$, there is a strongly connected orientation $D$ of $G$ satisfying that  $d_D^+(u)-d_D^-(u)\equiv\beta(u)\pmod{3}, \forall u\in V(G)$. In fact, a stronger form of Theorem \ref{THM:2tri-trees} is proved in Section \ref{sect:22trees} that for any graph $G$ with $|V(G)|\ge 4$ containing two edge-disjoint spanning triangle-trees, we have $G\in {\mathcal S}_3$.

Jaeger et al. \cite{JLPT92} proposed a conjecture, strengthening  Conjecture \ref{Tutteconj}, that every $5$-edge-connected graph is $\Z_3$-connected. Theorem \ref{triconn3flow} of Fan et al. \cite{triconZ_3} also has a form on $\Z_3$-group connectivity that, for any triangularly-connected graph $G$, $G\notin\langle \Z_3\rangle$ if and only if $G$ is constructed from $2$-sums of triangles and odd wheels.  Our $\Z_3$-group connectivity version of Theorem \ref{mainthm3flow} has a similar feature, but plus a bull-growing operation.

\begin{theorem}\label{tri-tree_character}
Let $G$ be a graph with a spanning triangle-tree. Then $G\notin\langle \Z_3\rangle$ if and only if $G$ can be constructed by one of the following operations:

$(i)$ \ $G$ is $K_3$ or $K_4$.

$(ii)$ \ $G=K_3\bigoplus_2 G_1$, where $G_1\notin\langle \Z_3\rangle$ contains a spanning triangle-tree.

$(iii)$ \ $G = {\mathcal B} \biguplus H$, where $H\notin\langle \Z_3\rangle$ contains a spanning triangle-tree.
\end{theorem}

Theorem \ref{tri-tree_character} also verifies the conjecture of Jaeger et al. \cite{JLPT92} in a strong sense that $4$-edge-connectivity suffices for $\Z_3$-connectedness on graphs containing a spanning triangle-tree.

A {\bf crystal} is a graph consisting of a triangle-path plus an extra edge connecting two leaves of the triangle-path. For instance, a wheel is a crystal by definition, and some more examples are depicted in Fig. \ref{Fig:3coloring}. Crystals are special graphs containing a spanning triangle-tree, and also play a role in our proofs. We obtain the following characterization of crystals as corollaries of Theorems \ref{mainthm3flow} and \ref{tri-tree_character}, connecting flows and vertex-coloring of crystals.

\begin{corollary}\label{Cor:crystal}
(i) \ A crystal has no $3$-NZF if and only if every vertex is of odd degree.

(ii) \ A crystal is $\Z_3$-connected if and only if it is vertex-$3$-colorable.
\end{corollary}

\section{Basic Lemmas and Bull-growing Operation}

We start with some basic lemmas, most of which have been widely used in flow theory. The following complete family properties were obtained in \cite{groupconnLai_complefamil} for $\setZTHR$ and in \cite{flowindex} for $\setCZ$.

\begin{lemma}\label{completefamily}\cite{groupconnLai_complefamil}\cite{flowindex} Let $\F\in \{\setZTHR, \setCZ\}$. Then each of the following holds.

(i) \ $K_1\in \F$.

(ii) \ If $e\in E(G)$ and $G\in \F$, then $G/e\in \F$.

(iii) \ If $H, G/H\in \F$, then $G\in \F$.

(iv) \ $2K_2\in \setZTHR$ and $4K_2\in \setCZ$.
\end{lemma}

The lifting lemma below on flows is routine to verify by definitions, as observed in \cite{Lai-survey,complement}. When $va, vb\in E_G(v)$, let $G_{[v,ab]}=G-va-vb+ab$ denote the graph obtained from $G$ by lifting $va,vb$ to become $ab$.

\begin{lemma}\cite{Lai-survey}\cite{complement}\label{splitZ_3}
Let $v$ be a $4^+$-vertex of a graph $G$ with $va, vb\in E_G(v)$.

(i) \ If $G_{[v,ab]}\in \setZTHR$, then $G\in \setZTHR$.

(ii) \ If $G_{[v,ab]}$ has a $3$-NZF, then so does $G$.

(iii) \ If $G_{[v,ab]}\in \mathcal S_3$, then so does $G$.

(iv) \ If $G-v+ab\in \mathcal S_3$, then so does $G$.
\end{lemma}

By repeatedly applying Lemma \ref{splitZ_3}(i), we immediately obtain the following more general lifting lemma, which will be a useful tool in our proofs.

\begin{lemma}\label{lemSplitPath}
Let $P$ be a path from $u$ to $v$ in $G$. If $G-E(P)+uv\in \setZTHR$, then $G\in \setZTHR$.
\end{lemma}

We refer to this operation as {\it lifting $E(P)$ in $G$ to become a new edge $uv$}.

In a tree $T$, for any $u,v\in V(T)$ there is a unique $uv$-path from $u$ to $v$, denoted by $P_{uv}$. A $uwv$-path means a path from $u$ to $v$ which goes through $w$, denoted by $P_{uwv}$. Fix a triangle-tree $\T$ and let $x,y\in V(\T)\cup E(\T)$ be two nonadjacent elements. Then there is a unique $xy$-triangle-path, denoted by $\TP(x,y,\T)$. We write $\TP(x,y)$ for convenience if no confusion occurs.

\begin{lemma}\label{tree+}
Let $G$ be a graph containing a spanning triangle-tree ${\mathcal T}={\mathcal T}(x_1,x_2,\ldots,x_n)$, where $x_1$ is a leaf of $\T$.

(i) \ For any $j,k>1$, the graph ${\mathcal T}+x_1x_j+x_1x_k$ is
$\Z_3$-connected.

(ii) \ Let $u,v,w\in V({\mathcal T})$. If $w\notin V(\TP(u,v,\T))$, then the graph ${\mathcal T}+uw+vw$ is $\Z_3$-connected.

(iii) \ If $G-\T$ contains a cycle, then $G\in\setZTHR$.
\end{lemma}

\begin{proof}
(i) \ Since $x_1x_2x_3$ is a triangle in $H={\mathcal T}+x_1x_j+x_1x_k$, we lift $x_1x_2, x_1x_3$ to obtain a graph $H_{[x_1,x_2x_3]}$ which contains parallel edges $x_2x_3$. Applying Lemma \ref{completefamily}(iii),(iv) to contract $2$-cycles consecutively along ${\mathcal T}-x_1$, we obtain a $2K_2\in\setZTHR$ which consists of the edges $x_1x_j,x_1x_k$. Hence,  $H_{[x_1,x_2x_3]}\in\setZTHR$, and so $H\in\setZTHR$ by Lemma \ref{splitZ_3}(i).

(ii) \ Since $w$ is not in $\TP(u,v,\T)$, in $\T$ there is a shortest triangle-path $\TP$ from $w$ to an edge in $\TP(u,v,\T)$ among all possible choices. Then $\TP(u,v)\cup \TP$ is a triangle-tree, where $w$ is a leaf of it. Set $H=\TP(u,v)\cup \TP+uw+vw$. Then $H\in\setZTHR$ by Lemma \ref{tree+}(i). In ${\mathcal T}+uw+vw$, we contract $H$ and then contract the resulting $2$-cycles consecutively, it eventually results in a $K_1$. Hence, ${\mathcal T}+uw+vw\in\setZTHR$ by Lemma \ref{completefamily}(iii).  Note that the Lemma also holds when $u=v$, in which case we can choose any triangle containing $u$ as $\TP(u,v,\T)$.

(iii) \ Let $C$ be a cycle of $G-\T$. If $V(C)=2$, there is a $2$-cycle $uw$ of $G$. Then Lemma \ref{tree+}(ii) is applied with $u=v$, and so ${\mathcal T}+uw+uw$ is $\Z_3$-connected.

If $V(C)\geq 3$, suppose $u,v,w\in V(C)$, and $E(C)$ consists of three edge-disjoint paths $P_{uv}, P_{vw}, P_{wu}$ in the cyclic order. There is a unique triangle-path $\TP(u,v,\T)$ since $\T$ is a spanning triangle-tree. If $w\notin V(\TP(u,v,\T))$, then we lift $P_{vw}, P_{wu}$ to become two edges $vw,uw$, and $\T+vw+uw\in\setZTHR$ by Lemma \ref{tree+}(ii). Thus, $G\in\setZTHR$ by Lemma \ref{lemSplitPath}. If  $w\in V(\TP(u,v,\T))$, then we must have $u\notin V(\TP(w,v,\T))$. In this case we lift $P_{vu}, P_{wu}$ to become two edges $vu,wu$, Hence,  $\T+vu+wu\in\setZTHR$ by Lemma \ref{tree+}(ii), and so $G\in\setZTHR$ by Lemma \ref{lemSplitPath} again.
\end{proof}

Note that, if any adding edges in Lemma \ref{tree+} (i) and (ii) are replaced by corresponding paths connecting the end vertices, we still get $\Z_3$-connected graphs by Lemma \ref{lemSplitPath}.
From Lemma \ref{tree+}, we also obtain the following corollary by applying Lemma \ref{completefamily} to contract $\Z_3$-connected subgraphs.

\begin{corollary}\label{Cor:spantreesubInZ_3}
Let $G$ be a graph with a spanning triangle-tree. Then $G\in \setZTHR$  if and only if it contains a nontrivial $\Z_3$-connected subgraph.
\end{corollary}

\begin{proof}
Let $H$ be a nontrivial $\Z_3$-connected subgraph and $\T$ a spanning triangle-tree of $G$. If $E(\T)\cap E(H)\neq \emptyset$, then in $G$ we contract the $\Z_3$-connected subgraph $H$ and then repeatedly contract $2$-cycles to eventually get a singleton $K_1$. Thus, $G\in \setZTHR$ by Lemma \ref{completefamily}(iii). Otherwise, $E(\T)\cap E(H)= \emptyset$. Since a $\Z_3$-connected graph must be $2$-edge-connected, $H$ contains a cycle which is edge-disjoint with the spanning triangle-tree $\T$ of $G$. Hence, $G\in\setZTHR$ by Lemma \ref{tree+}(iii).
\end{proof}

Now we present the bull-growing operation as a key tool in our later proofs.

\begin{lemma}\label{lemma3reduce}
Let $G={\mathcal B}\biguplus G_1$. The following statements hold.

(i) \ $G$ has a $3$-NZF if and only if $G_1$ has a $3$-NZF.

(ii) \ If $G\in\setZTHR$, then $G_1\in \setZTHR$. Conversely, if $G_1\notin \setZTHR$, then $G\notin\setZTHR$.
\end{lemma}

\begin{proof}
We adopt the notation as in the definition of bull-growing operation. Let $G_1=G-u-v+ab$, where $u,v$ are two adjacent $3$-vertices with a common neighbor $w$.

(i) is obvious and we shall only prove (ii). In fact, (i) follows from a similar argument below by replacing $\beta_1$-boundary with a zero-boundary. One may also see that the path $auvb$ of $G$ plays the same role as the edge $ab$ of $G_1$ in a mod $3$-orientation and the process can be reversed as well.

(ii) \ We shall prove $G_1\in \setZTHR$ by definition. Let $\beta_1$ be a $\Z_3$-boundary of $G_1$. Define $\beta: V(G)\rightarrow \Z_3$ as follows:
$$\left\{
\begin{aligned}
\beta(u)&=\beta(v)=0, \\
\beta(x)&=\beta_1(x) , \forall x\notin\{u,v\}.
\end{aligned}
\right.
$$
Since $\sum_{t\in V(G)} \beta(t) =\sum_{x\in V(G_1)} \beta_1(x)\equiv0 \pmod3$, $\beta$ is a $\Z_3$-boundary of $G$. As $G\in\setZTHR$, $G$ has an orientation $D$ such that $d_D^+(x)-d_D^-(x)\equiv\beta(x)\pmod3, \forall x\in V(G)$. Since $\beta(u)=\beta(v)=0$ and $u,v$ are adjacent, one of $u,v$ is oriented as all ingoing and the other is oriented as all outgoing. Thus $uw$ and $vw$ receive opposite orientations in $D$. Moreover, the edges $au,vb$ are either oriented from $a$ to $u$ and from $v$ to $b$, or all receive opposite directions.  So, we can orient $ab$ the same as $au$ and keep the orientations of the other edges of $G_1$ the same as $D$. Then this gives an orientation $D_1$ of $G_1$ with $d_{D_1}^+(y)-d_{D_1}^-(y)\equiv\beta_1(y)\pmod3,  \forall y\in V(G_1)$. So, $G_1\in\setZTHR$ by definition.
\end{proof}

The reverse of Lemma \ref{lemma3reduce} (ii) is not true in general, for example, it fails when $G_1$ is an odd wheel (and $a\neq b$ in bull-growing). However, when $G$ contains a spanning triangle-tree, Lemma \ref{lemma3reduce} can be strengthened to both necessary and sufficient.

\begin{lemma}\label{lemma3reduceTRItree}
Let $G$ be a graph with a spanning triangle-tree and $G=B \biguplus G_1$. Then $G\in\setZTHR$ if and only if $G_1\in \setZTHR$.
\end{lemma}

\begin{proof}
We still adopt the same notation as above and let $G_1=G-u-v+ab$. Since $G$ has a spanning triangle-tree $\T$, at least one of the edges of $\T$ must be in $\{aw,bw\}$, say $bw\in E(\T)$. We will show below that $G_1\in \setZTHR$ implies $G\in \setZTHR$.

Let $\beta:V(G)\rightarrow \Z_3$ be a $\Z_3$-boundary of $G$. If $\beta(u)\neq 0$, we lift $uw,uv$ to become a new edge $vw$, and then delete the vertex $u$ and the edge $ua$. Let $H$ be the resulting graph with corresponding boundary $\beta_1$, where $\beta_1(a)=\beta(a)+\beta(u)$ and $\beta_1(z)=\beta(z),\forall z\in V(G)\setminus\{u,a\}$. Then $H$ contains a $\Z_3$-connected subgraph $2K_2$ which consists of two parallel edges $vw$. By Corollary \ref{Cor:spantreesubInZ_3}, we have $H\in\setZTHR$, and so $H$ has an orientation $D_1$ satisfying boundary $\beta_1$.  We orient $ua$ to satisfy $\beta(u)$ and add $vu,uw$  back with their orientations kept as the lifted edge $vw$ of $D_1$. Specifically, we orient $ua$ from $u$ to $a$ if $\beta(u)=1$, and orient it from $a$ to $u$ if $\beta(u)=-1$.  This provides an orientation of $G$ satisfying boundary $\beta$.

If $\beta(v)\neq 0$, a similar argument applies. We lift $vb,vw$ to become a new edge $bw$ and delete the vertex $v$ and edge $uv$. Let $H$ be the resulting graph with corresponding boundary $\beta_1$ defined similarly. Then $H-u$ contains a triangle-tree with parallel edges $bw$, and so $H-u\in\setZTHR$ by Corollary \ref{Cor:spantreesubInZ_3}. By Lemma \ref{completefamily}(iii), $H\in\setZTHR$. Then we shall obtain an orientation of $G$ satisfying boundary $\beta$ similar as in the case $\beta(u)\neq 0$ above.

If $\beta(u)=\beta(v)=0$, we define a $\Z_3$-boundary $\beta_1$ of $G_1$ as $\beta_1(x)=\beta(x)$ for any $x\in V(G)\setminus\{u,v\}$. Since $G_1\in \setZTHR$, there is an orientation $D_1$ of $G_1$ satisfying $\beta_1$, where we may assume that the edge $ab$ is oriented from $a$ to $b$ (the other case is similar).
Then, in $G$ we keep the orientation of $E(G_1)-ab$ as in $D_1$, and orient the rest of edges as all ingoing to $u$ and outgoing to $v$. This gives an orientation of $G$ satisfying boundary $\beta$ as well. Therefore, $G$ is $\Z_3$-connected by definition.
\end{proof}

Note that in the bull-reduction operation, the condition that $G$ has a spanning triangle-tree $\T$  cannot ensure that $G_1$ contains a spanning triangle-tree. But if $u$ or $v$ is a leaf of $\T$, then the bull-reduction results in $G_1$ containing a spanning triangle-tree. In the proof below, we shall always apply this operation for leaves of spanning triangle-trees implicitly.

\begin{lemma}\cite{triconZ_3}
Let $G=H_1\bigoplus_2 H_2$. If $H_1\notin\setZTHR$ and $H_2\notin\setZTHR$, then $G\notin\setZTHR$.
\label{LEM: 2-sum}
\end{lemma}

\section{Graphs with Spanning Triangle-trees}

Now we are ready to prove our main results, Theorems \ref{tri-tree_character} and \ref{mainthm3flow}, for graphs containing a spanning triangle-tree.

{\bf Proof of Theorem \ref{tri-tree_character}:}
If $G$ satisfies one of (i), (ii) and (iii), then $G\notin \setZTHR$ by Lemmas \ref{lemma3reduceTRItree} and \ref{LEM: 2-sum}.
Now suppose that $G$ satisfies none of (i),(ii) or (iii). We shall show that $G\in \setZTHR$ by contradiction. Let $G$ be a minimum counterexample of Theorem \ref{tri-tree_character} with respect to $|E(G)|+|V(G)|$. Let $\T$ be a spanning triangle-tree of $G$. It is clear that for any vertex $v\in V(G)$, $d(v)\geq 3$. Otherwise, $G$ satisfies condition (ii).

Suppose $\TP=\TP(u,v)$ is a longest triangle-path among all possible triangle-paths in $G$.
Let $a$, $b$ be the neighbors of $u$ on $\TP$, where $a$ is a vertex with exactly $3$ neighbors in $\TP$.

We first claim that
\begin{equation}\label{CL:A}
E(\T)\cup E(\TP)\neq \emptyset.
\end{equation}

It is clear that $\TP$ contains a cycle. If no edge of $\TP$ is in  $E({\mathcal T})$, then by Lemma \ref{tree+}(iii) we have $\TP+{\mathcal T}\in \setZTHR$, and so $G\in\setZTHR$ by Corollary \ref{Cor:spantreesubInZ_3}. So, there is an edge of $\TP$ in  $E({\mathcal T})$, and (\ref{CL:A}) holds.

Thus, for any vertex $t\in V(G)\setminus V(\TP)$, there is a triangle-path $\TP(t,e)$ from $t$ to some $e\in E(\TP)$ by (\ref{CL:A}). Denote by $\TP(t,e_t)$ the shortest path among all triangle-paths  $\TP(t,e)$ with $e\in E(\TP)$. Note that $e_t\notin \{ua,ub\}$; otherwise, there is a longer triangle-path in $G$. If $t\in V(\TP)$, we also define $e_t=\emptyset$ and $\TP(x,e_t)=\emptyset$ for technical reasons.

Next, we show the following statement:
\begin{equation}\label{CL:B}
\text{$d_G(u)=3$ and $u$ is a leaf of ${\mathcal T}$.}
\end{equation}

Since $G$ does not satisfy (ii), $d_G(u)\neq 2$. Suppose, by contradiction, that $d_G(u)\geq 4$, and $s$, $d$ are two neighbors of $u$ other than $a$, $b$. Let $H=\TP\cup\TP(s,e_s)\cup\TP(d,e_d)$. Then $H$ is a triangle-tree, and moreover, $u$ is a leaf of $H$. Thus,  $H+uc+ud\in\setZTHR$ by Lemma \ref{tree+}(i), and so $G\in\setZTHR$ by Corollary \ref{Cor:spantreesubInZ_3}, which is a contradiction. So,  $d_G(u)=3$ and $u$ is a leaf of ${\mathcal T}$ since $\TP=\TP(u,v)$ is the longest triangle-path in $G$. This proves (\ref{CL:B}).

Let $x$ be the third neighbor of $u$, other than $a,b$, and let ${\mathcal Q}=\TP(x,e_x)$. Then we have $e_x\notin \{ $ab$,$ac$\}$. Otherwise, there is a longer triangle-path of $G$.

Let $G'=G_{[a,bc]}=G-ab-ac+bc$, and let $H$ be a maximum $\setZTHR$-subgraph of $G'$ containing $bc$. Since $bc$ is a 2-cycle, by Lemma \ref{completefamily}(iii) we contract $2$-cycles consecutively to obtain that $G'[V(\TP\cup {\mathcal Q})-a]\in\setZTHR$, and so $$V(\TP\cup {\mathcal Q})-a\subset V(H).$$

If $d_G(a)=3$, then by (\ref{CL:B}) the bull-reduction in (iii) is applied for $G$, and the resulting graph still has a spanning triangle-tree, a contradiction. Hence, $d_G(a)\geq 4$. Now we claim that
\begin{equation}\label{CL:C}
\text{there is a neighbor $y$ of $a$ that is not in $V(H)$.}
\end{equation}
Since $d_G(a)\geq 4$ and $a$ has exactly $3$ neighbors in $\TP$, we may let $y$ be a neighbor of $a$ not in $V(\TP)$. If $y\in V(H)$, then there are at least two neighbors of $a$, namely $u$ and $y$, in $V(H)$. By the maximality of $H$ and Lemma \ref{completefamily}(iii),(iv), we have $y\in V(H)$. Thus by Lemma \ref{completefamily}(iii) again, it follows from $u,y\in V(H)$  that $a\in V(H)$. Now we conclude that $V(\TP\cup {\mathcal Q})\subset V(H).$  Applying Lemma \ref{splitZ_3}(i), we also have $G[V(H)]\in\setZTHR$, and so $G\in\setZTHR$ by Corollary \ref{Cor:spantreesubInZ_3},  a contradiction. This verifies (\ref{CL:C}).

Since $d_G(y)\geq 3$ and by Lemma \ref{completefamily}(iii), at most one neighbor of $y$ is in $V(H)$, and so there is a neighbor $z$ of $y$ not in $V(H)$. This also means that $\TP(z,e_z)$ must intersect $\TP$ at $ab$ or $ac$, w.l.o.g., say $e_z=ac$. Otherwise, we have $z\in V(H)$, and so $y\in V(H)$ by Lemma \ref{completefamily}(iii), a contradiction.

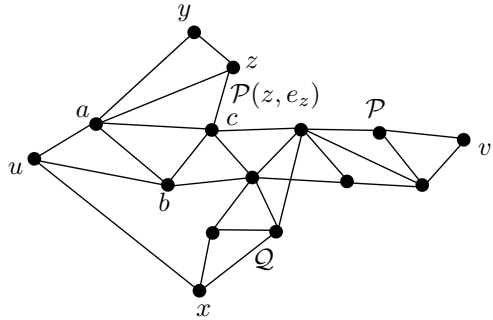
\begin{figure}[!hpbt]
    \centering
    \includegraphics[width=0.4\textwidth,clip]{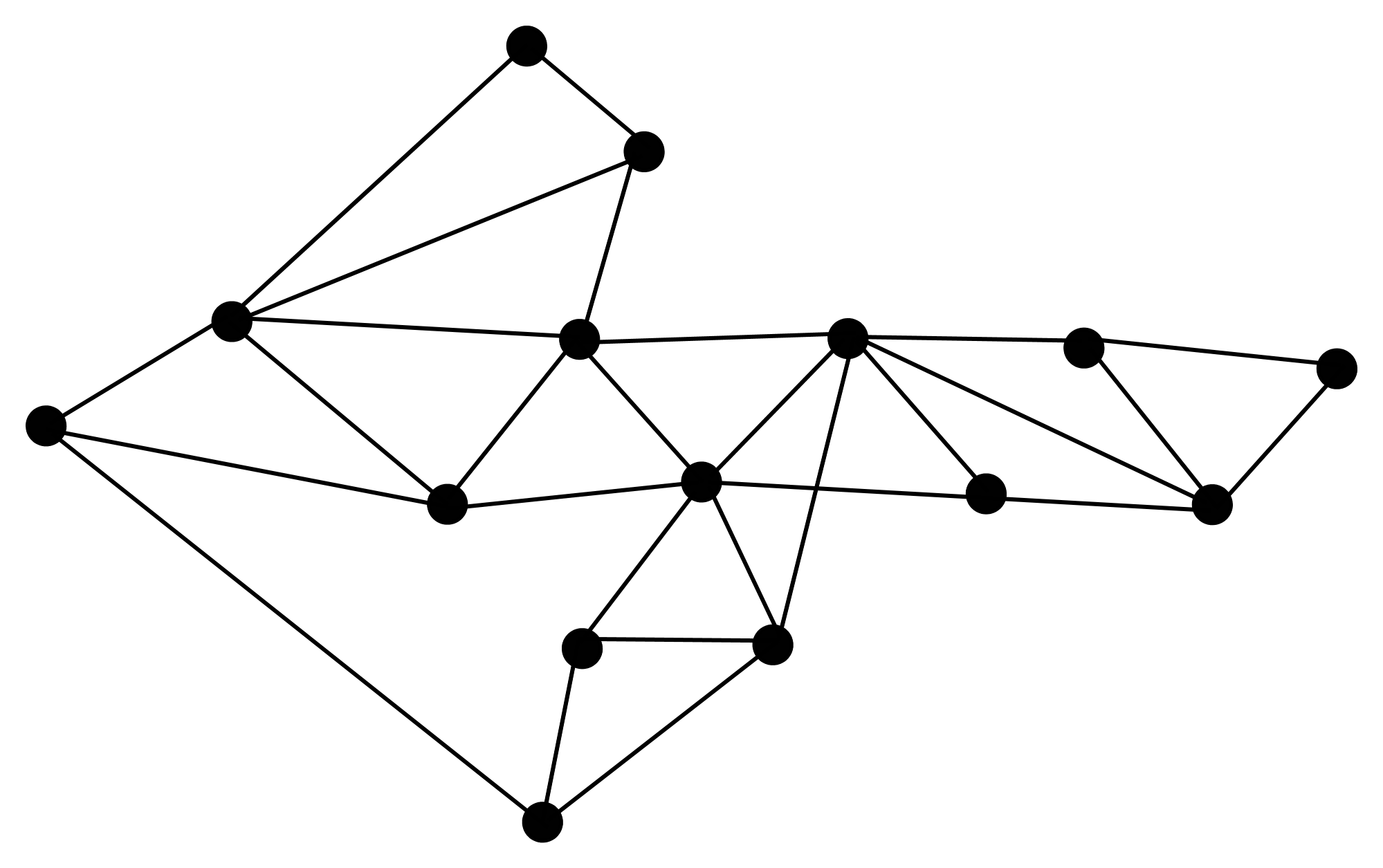}
    \put(-12.5,3.5){$u$}
    \put(-10.7,5){$a$}
    \put(-8.5,2.5){$b$}
    \put(-6.7,4.8){$c$}
    \put(-6.2,6.3){$z$}
    \put(-6.6,5.4){$\TP(z,e_z)$}
    \put(-8,7.7){$y$}
    \put(0,4){$v$}
    \put(-7.5,-0.3){$x$}
    \put(-6,1){${\mathcal Q}$}
    \put(-3,5){$\TP$}

\caption{\small\it A longer triangle-path.}\label{longerpath}
\end{figure}

{\bf The final step.} If $\TP(z,e_z)$ is a triangle $acz$, see Fig. \ref{longerpath}, then $\TP-u+za+zc+ya+yz$ is a longer triangle-path of $G$, a contradiction. Otherwise, $\TP{(z,e_z)}$ contains at least two triangles, and so $\TP-u+\TP(z,e_z)$ is a triangle-path longer than $\TP$, again a contradiction to the maximality of $\TP$. This finishes the proof.
\qed

{\bf Proof of Theorem \ref{mainthm3flow}:}
If $G$ is formed from $K_4$ by a series of bull-growing operations, then it has no $3$-NZF by Lemma \ref{lemma3reduce}. Conversely, assume that $G$ has no $3$-NZF. Then, $G\notin \setZTHR$. We apply Theorem \ref{tri-tree_character} on $G$.

Suppose $G=K_3\bigoplus_2 G_1$, where $G_1$ contains a spanning triangle-tree $\T$. Let $abc$ correspond to the $K_3$ in the $2$-sum, where $a$ is a $2$-vertex of $G$. Then $G_{[a,bc]}$ contains a $2$-cycle $bc$, which shows $G_{[a,bc]}\in\setZTHR$ by Corollary \ref{Cor:spantreesubInZ_3}, and therefore, has a $3$-NZF. Hence $G$ has a $3$-NZF by Lemma \ref{splitZ_3}(ii), a contradiction.

Now suppose $G={\mathcal B}\biguplus G_1$, where $G_1$ contains a spanning triangle-tree $\T$. By Lemma \ref{lemma3reduce}, $G_1$ has no $3$-NZF if and only if $G$ has no $3$-NZF. This proves Theorem \ref{mainthm3flow}.
\qed

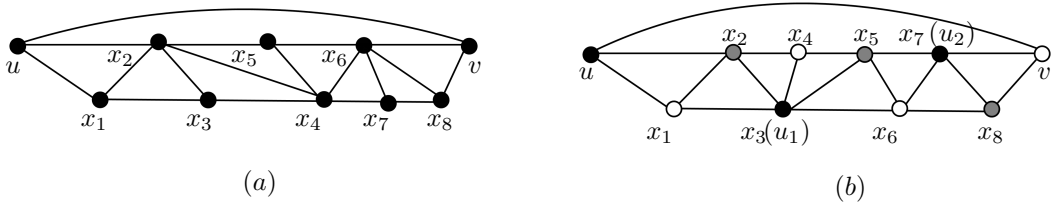
\begin{figure}[!hpbt]
\minipage{0.5\textwidth}
    \centering
    \includegraphics[width=0.85\textwidth,,clip]{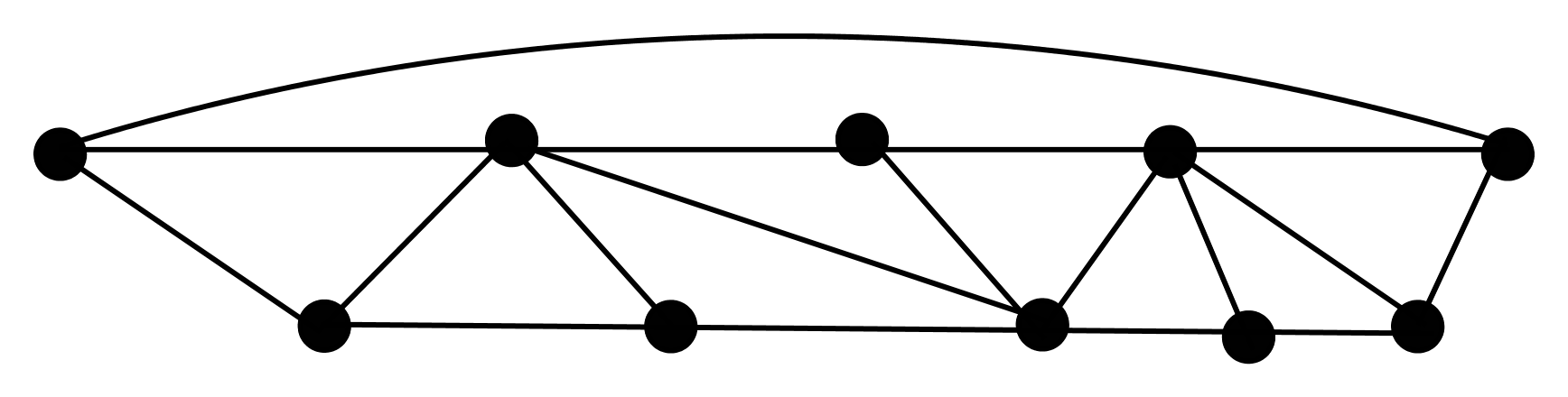}
    \put(-12.8,1.3){$u$}
    \put(-10.8,-0.1){$x_1$}
    \put(-10.1,1.6){$x_2$}
    \put(-8,-0.1){$x_3$}
    \put(-5.1,-0.1){$x_4$}
    \put(-6.8,1.6){$x_5$}
    \put(-4.4,1.6){$x_6$}
    \put(-3.3,-0.1){$x_7$}
    \put(-1.6,-0.1){$x_8$}
    \put(-0.5,1.3){$v$}
    \put(-6.5,-1.8){$(a)$}
\endminipage\hfill
\minipage{0.5\textwidth}
    \centering
    \includegraphics[width=0.85\textwidth,,clip]{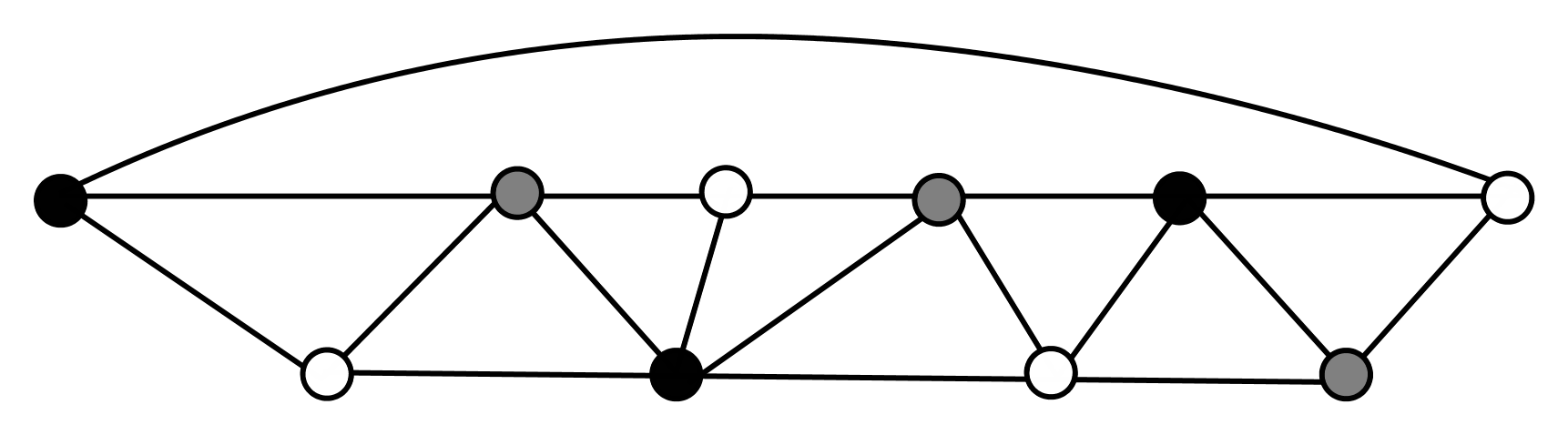}
    \put(-12.8,1.3){$u$}
    \put(-11,-0.3){$x_1$}
    \put(-9,2.3){$x_2$}
    \put(-8.5,-0.3){$x_3$}
    \put(-7.9,-0.3){$(u_1)$}
    \put(-5,-0.3){$x_6$}
    \put(-7.25,2.3){$x_4$}
    \put(-5.5,2.3){$x_5$}
    \put(-2.2,-0.3){$x_8$}
    \put(-4.3,2.3){$x_7$}
    \put(-3.5,2.3){$(u_2)$}
    \put(-0.6,1.3){$v$}
    \put(-6,-1.8){$(b)$}
\endminipage
    \caption{\it The crystals in Corollary \ref{Cor:crystal}. }
    \label{Fig:3coloring}
\end{figure}

{\bf Proof of Corollary \ref{Cor:crystal}:}
Let ${\mathcal C}=\TP(u,v) +uv$ be a crystal, where the vertices of ${\mathcal C} $ are ordered as $u,x_1,x_2,\cdots, x_k,v$ and $d_{{\mathcal C}}(x_1)=3$. When $|V({\mathcal C})|\leq 5$, ${\mathcal C}$ is a wheel and the statements clearly hold. Now we proceed by induction and assume $|V({\mathcal C})|\geq 6$.

(i) \ By Theorem \ref{mainthm3flow}, ${\mathcal C}$ has no $3$-NZF if and only if it is formed from $K_4$ by a series of bull-growing operations. Since bull-growing operation keeps the parity of degree of each vertex, that ${\mathcal C}$ has no $3$-NZF would imply that each vertex has odd degree. On the other hand, if each vertex of ${\mathcal C}$  is of odd degree, then we have that $d_{{\mathcal C}}(x_1)$, $d_{{\mathcal C}}(x_2)$ and $d_{{\mathcal C}}(x_3)$ are odd (see Fig. \ref{Fig:3coloring}(a)). Thus, $d_{{\mathcal C}}(x_1)=3$ and at least one of $d_{{\mathcal C}}(x_2), d_{{\mathcal C}}(x_3)$ is also $3$. Without loss of generality, we assume $d_{{\mathcal C}}(x_3)=3$. And $x_2$ is a common neighbor of $u$ and $x_1$.
Hence, ${\mathcal C}={\mathcal B}\biguplus(\TP(x_3,v)+x_3v)$. Now $\TP(x_3,v)+x_3v$ is smaller than ${\mathcal C}$ and each vertex of it has odd degree. Thus $\TP(x_3,v)+x_3v$ has no $3$-NZF by induction, and so ${\mathcal C}$ has no $3$-NZF by Lemmas \ref{lemma3reduce} and \ref{LEM: 2-sum}.

(ii) \ Let $\psi:V(\TP(u,v))\rightarrow \{\text{black, white, gray}\}$ be a proper $3$-coloring of $\TP(u,v)$ with $\psi(u)=black$, and let $u_{1}$ be the first vertex of $x_1,x_2,\cdots, x_k,v$ with color black, w.l.o.g., say $d_{{\mathcal C}}(x_1)=3$ and $u_1=x_3$ (see Fig. \ref{Fig:3coloring} (b)). Then $G_1={\mathcal C}-u-x_1+x_3v$ is the bull-reduction of ${\mathcal C}$ and $H=\TP(x_3,v)+x_3v$ is a new crystal. Similar to (i), we have that either $G_1=H$ (in this case ${\mathcal C}={\mathcal B}\biguplus H$), or $G_1$ consists of $2$-sums of $H$ and triangles. By Lemmas \ref{lemma3reduceTRItree} and \ref{LEM: 2-sum}, $G_1\in \setZTHR$ if and only if $H\in \setZTHR$. By induction,  $H=\TP(x_3,v)+x_3v\in \setZTHR$ if and only if $H$ is vertex-$3$-colorable, i.e., $\psi(v)\neq black$. Hence by Lemma \ref{lemma3reduceTRItree}, ${\mathcal C}\in \setZTHR$ if and only if $\psi(v)\neq black$. Thus, (ii) holds, which completes the proof.
\qed

\section{Two Spanning Triangle-trees}
\label{sect:22trees}

An elementary theorem of Robbins \cite{Robbins} (or see Theorem 5.1 in \cite{bondy}) shows that every connected graph without cut edges has a strongly connected orientation. In fact, such a strongly connected orientation can be easily obtained from ear-decompositions. This motivates the following lemma.

\begin{lemma}\label{LEM: G1G2}
If $G$ can be edge-partitioned into two spanning subgraphs $G_1$ and $G_2$ such that $G_1\in\setZTHR$ and $G_2$ is $2$-edge-connected, then $G\in\setCZ$.
\end{lemma}

\begin{proof}
Let $\beta$ be a $\Z_3$-boundary of $G$. We first give $G_2$ a strongly connected orientation $D_2$ by Robbins' Theorem. Suppose that the boundary of $G_2$ corresponding to $D_2$ is $\beta_2$. Since $G_1\in\setZTHR$, there is a mod 3-orientation $D_1$ of $G_1$ for the $\Z_3$-boundary $\beta-\beta_2$. Since both $G_1$ and $G_2$ are spanning, $D=D_1\cup D_2$ is a strongly mod 3-orientation of $G$ for the boundary $\beta$. That is, for any $v\in V(G)$, $$d_D^+(v)-d_D^-(v)=(d_{D_2}^+(v)-d_{D_2}^-(v))+(d_{D_1}^+(v)-d_{D_1}^-(v))\equiv\beta_2(v)+(\beta(v)-\beta_2(v))\equiv\beta(v)\pmod3.$$ So, $G\in\setCZ$ by definition.
\end{proof}

Our strategy for the proof of Theorem \ref{THM:2tri-trees} is to apply some extreme choice to find a $2$-edge-connected spanning subgraph from one triangle-tree, and then get a $\Z_3$-connected spanning subgraph from another triangle-tree by adding some extra edges.
We will need one more lemmas before proving Theorem \ref{THM:2tri-trees}.

Let $\T$ be a triangle-tree. We say that an edge set $X$ of $E(\T)$ is {\em removable} if $\T-X$ is $2$-edge connected; each edge $e\in X$ is called a removable edge.

\begin{proposition}\label{PROP: removable}
Let $\T$ be a triangle-tree on $n\ge 4$ vertices with $t$ leaves. Then $\T$ contains a removable set of size at least $n-t-1$.
\end{proposition}

\begin{proof}
It is easy to check this fact for $|V(\T)|\leq 5$. Assume it holds for $|V(\T)|\leq k-1$. When $|V(\T)|= k$, let $v$ be the new vertex added such that $abv$ forms a new triangle. If neither $a$ nor $b$ is a leaf, then the removable set of $\T$ is the same as $\T-v$. If one of $a,b$ is a leaf, then the edge $ab$ is removable, and so the size of removable set increases. By induction, the proposition holds.
\end{proof}

\begin{theorem}\label{THM:S3twotriangletrees}
For any graph $G$ with $|V(G)|\ge 4$ containing two edge-disjoint spanning triangle-trees, we have $G\in {\mathcal S}_3$.
\end{theorem}

\begin{proof}
Suppose, to the contrary, that $G\notin {\mathcal S}_3$.
Let $\T_1$ and $\T_2$ be two edge-disjoint spanning triangle-trees of $G$. We will move some edges from $\T_1$ to $\T_2$ to obtain a $\Z_3$-connected graph. At the same time, we shall also keep the remaining part of $\T_1$ being $2$-edge-connected. Let $R_i$ be a largest removable set of $\T_i$ for $i=1,2$. We may also view $R_i=G[R_i]$ as an edge-induced subgraph of $G$. Without loss of generality, assume that
 \begin{equation*}
   |R_1|\ge |R_2|.
 \end{equation*}
Clearly, $\T_1-R_1$ is still $2$-edge-connected by definition.  Ultimately, we will show that
\begin{equation}\label{EQ:inZ3}
\T_2+R_1\in \setZTHR.
\end{equation}
Then it follows from Lemma \ref{LEM: G1G2} that $G\in{\mathcal S}_3$, a contradiction.

\begin{claim}\label{CL: R1tree}
The graph $R_1$  is a tree.
\end{claim}

\begin{proof}
If $R_1$ contains a cycle, then by Lemma \ref{tree+}(iii) we have $\T_2+R_1\in\setZTHR$. Hence, $G\in{\mathcal S}_3$ by Lemma \ref{LEM: G1G2}, a contradiction. Thus $R_1$ is acyclic. Let $L_1$ be the set of leaves in $\T_1$. Clearly, $L_1\cap V(R_1)=\emptyset$ since there is no removable edge incident to a leaf. Thus by Proposition \ref{PROP: removable}, we have $|R_1|\geq |V(G)|-|L_1|-1\geq |V(R_1)|-1$. As $R_1$ is acyclic, we conclude that it is a tree.
\end{proof}

\begin{claim}\label{CL: disPATH}
Let $u,w\in V(R_1)$. For any $v\in V(\TP(u,w,\T_2))\cap V(R_1)$, there is a $uvw$-path in $R_1$.
\end{claim}

\begin{proof}
By contradiction, assume that $v$ is not in the $uw$-path $P_{uw}$ of $R_1$. Since $R_1$ is a tree by Claim \ref{CL: R1tree}, there is a unique shortest path from $v$ to $P_{uw}$ in $R_1$, where the intersection vertex is denoted by $c$. Then we have three paths $P_{uc}$, $P_{vc}$, $P_{wc}$ intersecting at $c$. Note that it is possible that $c=u$ or $c=v$. Since $v\in V(\TP(u,w,\T_2))\cap V(R_1)$, $\TP(u,w,\T_2)$ is divided into two triangle-paths $\TP(u,v,\T_2)$ and $\TP(v,w,\T_2)$.
Moreover, we have either $c\notin V(\TP(u,v,\T_2))$ or $c\notin V(\TP(v,w,\T_2))$, or both. Assume, w.l.o.g., that $c\notin V(\TP(u,v,\T_2))$. We lift the two paths $P_{uc}$, $P_{vc}$ to become two new edges $uc, vc$. Then, $\T_2+uc+vc\in\setZTHR$ by Lemma \ref{tree+} (ii), and so $\T_2+P_{uc}+P_{vc}\in\setZTHR$ by Lemmas \ref{splitZ_3} and \ref{lemSplitPath}. Hence, $\T_2+R_1\in \setZTHR$, i.e., (\ref{EQ:inZ3}) holds, yielding to a contradiction.
\end{proof}

\begin{claim}\label{CL: edgedisjiont}
For any distinct edges $e_1=u_1v_1\in R_1$ and $e_2=u_2v_2\in R_1$, the triangle-paths $\TP(u_1,v_1,\T_2)$ and $\TP(u_2,v_2,\T_2)$ are edge-disjoint.
\end{claim}

\begin{proof} Assume it is not the case. Then $\T^*=\TP(u_1,v_1,\T_2) \cup \TP(u_2,v_2,\T_2)$ is a triangle-tree, which is a sub-triangle-tree of $\T_2$. Since $R_1$ is a tree by Claim \ref{CL: R1tree}, there is a shortest path connecting $e_1$ and $e_2$ in $R_1$. By possibly relabeling the vertices, we may denote this path by $P_{u_1u_2}$ from $u_1$ to $u_2$ in $R_1$. If $u_2\in V(\TP(u_1,v_1,\T_2))$, then by Claim \ref{CL: disPATH} there is a $u_1u_2v_1$-path $P_{u_1u_2v_1}$ in $R_1$. Thus $P_{u_1u_2v_1}+u_1v_1$ is a cycle in $R_1$, a contradiction to Claim \ref{CL: R1tree}. Hence we have $u_2\notin V(\TP(u_1,v_1,\T_2))$, and so $u_2$ is a leaf of $\T^*$. Now lift the path $P_{u_1u_2}$ to become a new edge $u_1u_2$. Then, $\T^*+u_1u_2+v_2u_2\in\setZTHR$ by Lemma \ref{tree+} (i). Thus, $\T+u_1u_2+v_2u_2\in\setZTHR$ and $\T+R_1\in\setZTHR$ by Lemmas \ref{splitZ_3}, \ref{lemSplitPath} and Corollary \ref{Cor:spantreesubInZ_3}. Thus, (\ref{EQ:inZ3}) holds and $G\in{\mathcal S}_3$, a contradiction.
\end{proof}

\begin{claim}\label{CL: R1=R2}
We have $|R_2|=|R_1|$, and for each $uv\in R_1$ the graph $\TP(u,v,\T_2)+uv$ is a $K_4$.
\end{claim}

\begin{proof}
Recall that we already have $|R_1|\ge |R_2|$ by the assumption in the beginning. It remains to show that $|R_2|\ge |R_1|$.  For each edge $e=uv\in R_1$, $\TP(u,v,\T_2)$ is a triangle-path with at least $4$ vertices, and so it contains at least one removable edge, namely the edge in the triangle containing $u$ but not incident to $u$.
Moreover, all those edges are distinct by Claim \ref{CL: edgedisjiont}. Let $R_2'$ be the collection of all such edges. Then, $|R_2'|\ge |R_1|$, and so by the maximality of $R_2$ we have $|R_2|\ge |R_2'|\ge |R_1|$.
Thus, $|R_2|=|R_1|$. Furthermore, if $\TP(u,v,\T_2)$ contains at least $5$ vertices for some $e=uv\in R_1$, then we can easily select two removable edges from it, namely the edge in the triangle containing $u$ but not incident to $u$ and also a similar edge for $v$.
This would result in $|R_2'|> |R_1|$, a contradiction. Hence we conclude that the graph $\TP(u,v,\T_2)+uv$ is exactly a $K_4$ for each $uv\in R_1$.
\end{proof}

\begin{claim}\label{CL: R2R1ge2}
We have $|V(G)|\ge 5$ and $|R_2|=|R_1|\ge 2$.
\end{claim}

\begin{proof}
When $V(G)=4$, it is easy to check that $G\in {\mathcal S_3}$ by Lemma \ref{LEM: G1G2}. Specifically, there are three non-isomorphic  distributions of $\T_1$ and $\T_2$, and $G$ can be edge-partitioned into a spanning $\Z_3$-connected subgraph and a spanning $2$-edge-connected subgraph in each case. An alternate method is to apply lifting techniques of Lemma \ref{splitZ_3} (iii), and see \cite{complement} for more details. Thus we have $|V(G)|\ge 5$.

Now suppose $|R_2|=|R_1|=1$. Then both $\T_1$ and $\T_2$ contain $|V(G)|-2$ leaves by Proposition \ref{PROP: removable}. In fact, this indicates that both $\T_1$ and $\T_2$ are isomorphic to the complete tripartite graph $K_{1,1,|V(G)|-2}$, called triangular-book. As $|V(G)|\ge 5$, there are at least $|V(G)|-4\ge 1$ common leaves for  $\T_1$ and $\T_2$. Let $x$ be a common leaf of $\T_1$ and $\T_2$, and let $xyz$ be the corresponding triangle in $\T_1$. Now consider the graph $G'=G-x+yz$. Then $G'$ contains two edge-disjoint spanning triangle-trees $\T_1'=\T_1-x$ and $\T_2'=\T_2-x$. Moreover, $\T_2'$ is $2$-edge-connected, and $\T_1'+yz\in\setZTHR$ since it contains parallel edges $yz$ and by Corollary \ref{Cor:spantreesubInZ_3}. Thus,  $G'=G-x+yz\in{\mathcal S}_3$ by Lemma \ref{LEM: G1G2}. Hence,  $G\in{\mathcal S}_3$ by Lemma \ref{splitZ_3} (iv), a contradiction.
\end{proof}

\begin{figure}[!hpbt]
    \minipage{0.5\textwidth}
    \centering
    \includegraphics[width=0.45\textwidth,trim=0 0 0 0,clip]{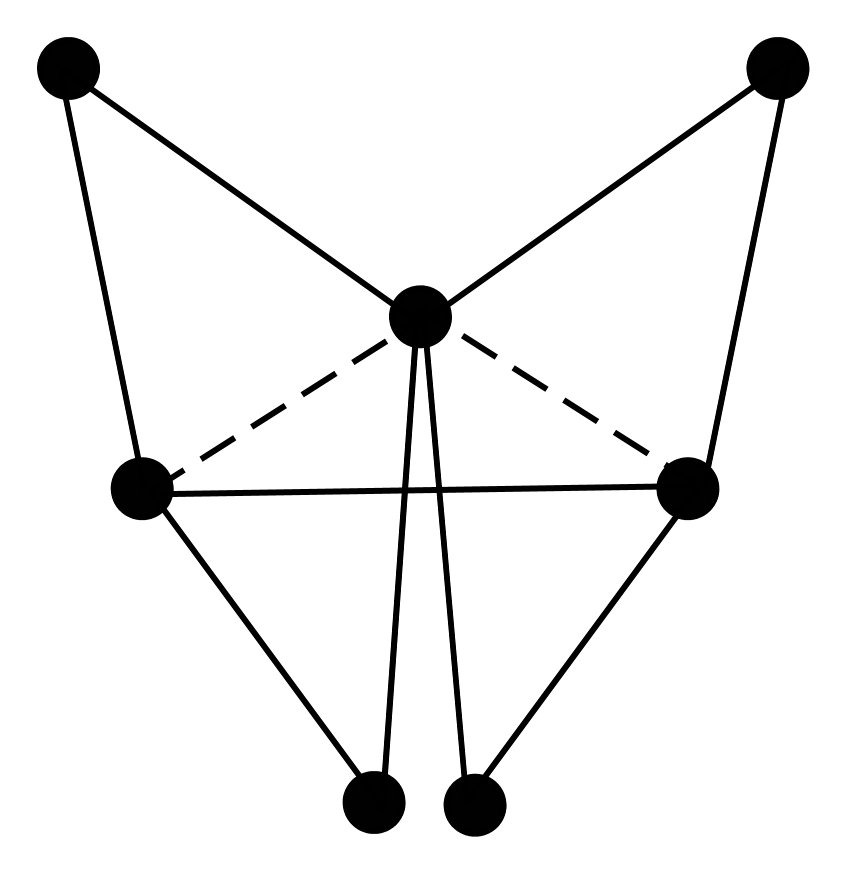}
    \put(-0.2,6){$u_t$}
    \put(-7.2,6){$u_k$}
    \put(-2.6,0){$v_t$}
    \put(-4.2,0){$v_k$}
    \put(-5.3,3.9){$f_k$}
    \put(-3.6,5){$w$}
    \put(-2.1,3.9){$f_t$}
    \put(-2.5,2.5){$e$}
    \put(-6,2.3){$u$}
    \put(-1.2,2.4){$v$}
    \put(-3.7,-1){$(1)$}
    \endminipage\hfill
    \minipage{0.5\textwidth}
    \centering
    \includegraphics[width=0.7\textwidth,trim=0 0 0 0,clip]{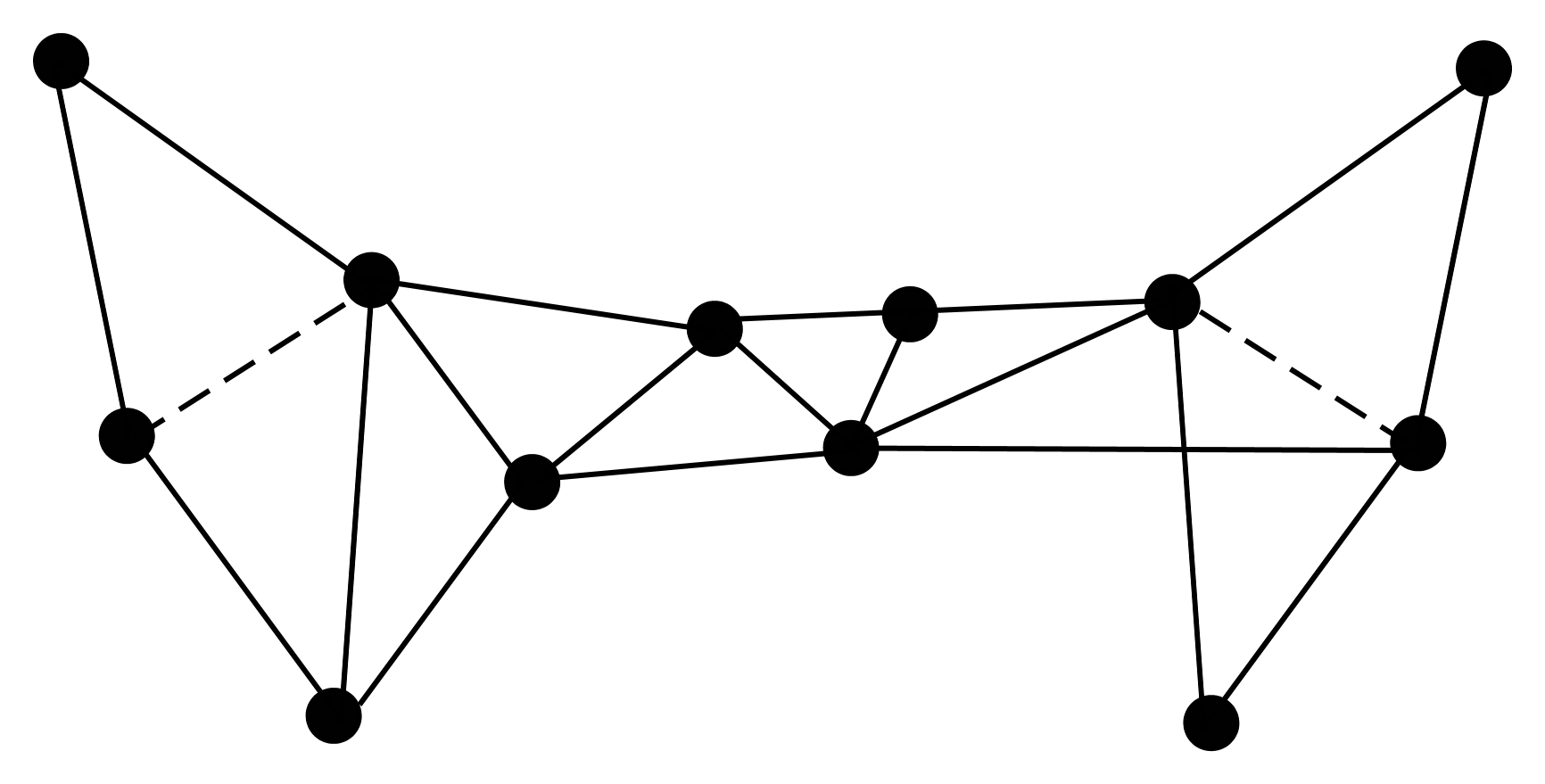}
    \put(0,5.2){$u_t$}
    \put(-11,5.5){$u_k$}
    \put(-2,0){$v_t$}
    \put(-9,0){$v_k$}
    \put(-9.6,3){$f_k$}
    \put(-1.6,3){$f_t$}
    \put(-5.9,3.5){$u$}
    \put(-7,1.5){$v$}
    \put(-5,-2){$(2)$}
    \endminipage
    \caption{\small The edge $uv$ is removable in the final step in the proof of Theorem \ref{THM:S3twotriangletrees}.}\label{Fig:MoreRemovableEdge}
\end{figure}
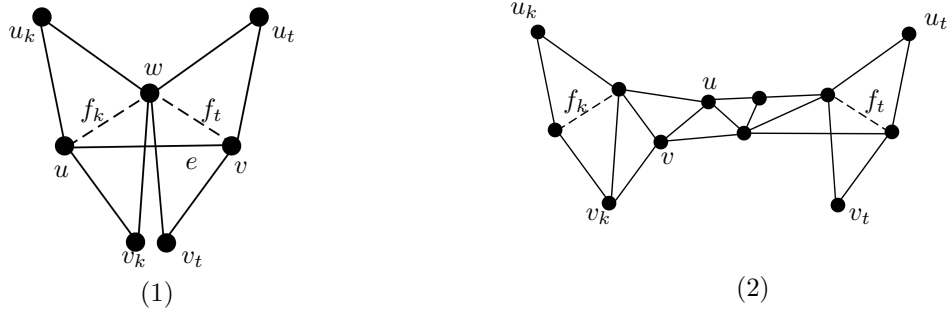

{\bf The final step.}
As in the proof of Claim \ref{CL: R1=R2}, let $R_2'$ be the collection of all edges $f$ such that $f=\TP(u,v,\T_2)-u-v$ for some $uv\in R_1$. Denote $R_2'=\{f_1,f_2,\cdots,f_s\}$, where $|R_1|=|R_2|=s$. Choose $\TP(f_k,f_t,\T_2)$ as small as possible among all possible distinct edges $f_k,f_t\in R_2'$.

Assume that $\TP(f_k,f_t,\T_2)$ is a triangle, say $uvw$, where $f_k=uw$ and $f_t=vw$. We further denote the corresponding $K_4$ associated with $f_k$ and $f_t$ by $u_kuv_kw$ and $u_tuv_tw$ (see Fig. \ref{Fig:MoreRemovableEdge}(1)). If $uv\in R_2'$, then $R_2'$ contains a cycle $uvw$, and so $\T_1+R_2'\in\setZTHR$ by Lemma \ref{tree+}(iii). Moreover, $\T_2-R_2'$ is still $2$-edge-connected as each vertex, including $u,v$, is still in a cycle. Thus it follows from Lemma \ref{LEM: G1G2} that $G\in{\mathcal S}_3$, a contradiction. So, we have $uv\notin R_2'$. Now let $R_2''=R_2'\cup\{uv\}$. Then $\T_2-R_2''$ is still $2$-edge-connected by the same reason, and so $R_2''$ is a removable set with size $|R_2''|=|R_2'|+1=s+1>s=|R_2|$, a contradiction to the maximality of $R_2$.

Assume instead that $\TP(f_k,f_t,\T_2)$ contains at least $4$ vertices. Let $C$ be the outer Hamiltonian cycle of $\TP(f_k,f_t,\T_2)$, where $f_k,f_t\in E(C)$. Then $C$ contains a chord $uv$ (see Fig. \ref{Fig:MoreRemovableEdge}(2)). By the minimality of $\TP(f_k,f_t,\T_2)$, we have $uv\notin R_2'$. Otherwise $\TP(f_k,uv,\T_2)$ causes a shorter triangle-path.  Now let $R_2''=R_2'\cup\{uv\}$. Then $\T_2-R_2''$ is still $2$-edge-connected since $u$ and $v$ are still contained in a cycle. Thus $R_2''$ is a removable set, but we have $|R_2''|=|R_2'|+1=s+1>s=|R_2|$, again a contradiction. This completes the proof.
\end{proof}

\section{Remarks on Triangularly-connected Subgraphs}

Recall that the group connectivity version of Theorem \ref{triconn3flow} of Fan et al \cite{triconZ_3} states as follows.

\begin{theorem}\label{Thm:TriConZ_3_2-SUM}
Let $G$ be a triangularly-connected graph with $|V(G)|\geq 3$. Then $G\notin \setZTHR$ if and only if there is a subgraph $G_1$ and an odd wheel or a triangle, called $W$, such that $G = W \bigoplus_2 G_1$, where $G_1\notin\setZTHR$ is triangularly-connected.
\end{theorem}

From this theorem, we can easily characterize triangularly-connected graphs without spanning triangle-trees under $\Z_3$-connectivity.
An {\em eccentrical edge} of a wheel is an edge that is not incident with the center vertex. A wheel in a graph $G$ is {\em fully $2$-summed} if for each eccentrical edge $e$, there exist subgraphs $A,B$ of $G$ such that $G=A\bigoplus_2 B$ and $E(A)\cap E(B)=\{e\}$ (see Fig. \ref{wheelstar} below).

\begin{figure}[!hpbt]
    \centering
    \includegraphics[width=0.5\textwidth,trim=0 0 0 0,clip]{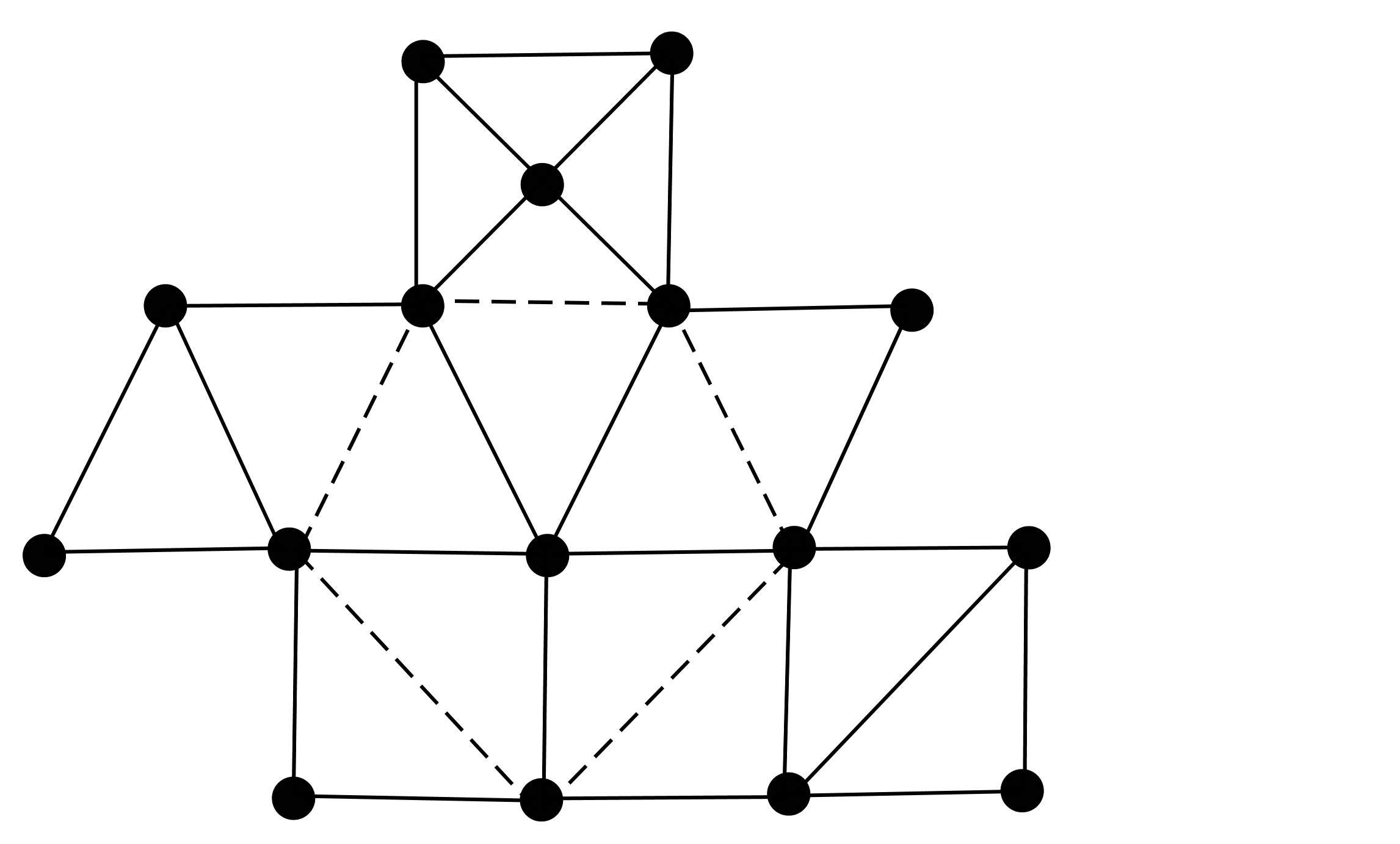}
    \caption{\small\it A wheel that is fully 2-summed.}\label{wheelstar}
\end{figure}

\begin{proposition}\label{thmwheelstar}
Let $G\notin \setZTHR$ be a triangularly-connected graph. Then $G$ has no spanning triangle-tree if and only if there is an odd wheel of $G$ that is fully $2$-summed.
\end{proposition}

\begin{proof}  The ``if" part is trivial, since each eccentrical edge of the fully $2$-summed odd wheel must be in the spanning triangle-tree, which leads to a contradiction. It remains to justify the ``only if" part.

Suppose, to the contrary, that $\T$ is a maximum triangle-tree of $G$, where $|V(\T)|<|V(G)|$.
Then there exists a pair of incident edges $e_1$,$e_2$ with $e_1\in E(\T)$, $e_2\notin E(\T)$, where $e_1$ and $e_2$ are intersecting at $v\in V(\T)$.
Since $G$ is triangularly-connected, there is a triangle-path $\TP$ from $e_1$ to $e_2$.
So, there must be a triangle with $2$ vertices in $V(\T)$, named $x$, $y$, and one vertex in $V(G)-V(\T)$, named $z$.
If $xy\in E(\T)$, then $\T+xz+yz$ is a larger triangle-tree, a contradiction. So, we have $xy\notin E(\T)$ and there is a triangle $xyt$ on $\TP$ with $t\in V(T)$.
If there is at most one edge of $xt, yt$ in $E(\T)$, say possibly $yt$. Then by Lemma \ref{splitZ_3} (i), $T+xy+xt\in \setZTHR$. Thus,  $G\in\setZTHR$ by Lemma \ref{completefamily} (iii). So, assume instead that both $xt,yt$ are in $E(\T)$. Since $\T$ is a triangle-tree, there is a triangle-path ${\mathcal Q}$ from $xt$ to $yt$. Moreover, ${\mathcal Q}$ is a fan, a wheel with one eccentrical edge deleted.
If there is an eccentrical edge $f$ not in any $2$-sum in $G-{\mathcal Q}$, then $\T-f+xy+xz+yz$ is a larger triangle-tree of $G$, a contradiction. So, $G$ has a fully $2$-summed wheel. The proof is thus complete.
\end{proof}

From Theorem \ref{Thm:TriConZ_3_2-SUM} and Proposition \ref{thmwheelstar}, non-$\Z_3$-connected triangularly-connected graphs  almost have the same structure as graphs containing spanning triangle-trees. Thus all the main results concerning spanning triangle-trees in this paper can be easily transferred to graphs containing spanning triangularly-connected subgraphs, with essentially the same proof. For example, we have the following more general theorem.

\begin{theorem}
Let $G$ be a graph containing a spanning triangularly-connected subgraph.

(a) \ $G$ has no $3$-NZF if and only if $G = {\mathcal B} \biguplus G_1$, where $G_1$ contains a spanning triangularly-connected subgraph and has no $3$-NZF. In other words, $G$ has no $3$-NZF if and only if $G$ is formed from $K_4$ by a series of bull-growing operations.

(b) \ $G\notin\langle \Z_3\rangle$ if and only if $G$ can be constructed from $K_3$ or $K_4$ by $2$-sum and bull-growing operations.
\end{theorem}

The methods developed in this paper may be helpful in studying the following more general problem.

\begin{problem}
Characterize the $3$-flow property of all graph $G$ such that for any $u,v\in V(G)$ there is $uv$-triangle-path in $G$.
\end{problem}

\end{document}